\documentclass{article}

\usepackage{graphicx}
\usepackage{amsmath,amssymb}
\usepackage{amsfonts}
\usepackage{amstext}
\usepackage{amsthm}
\usepackage{amsopn}
\usepackage{amsbsy}
\usepackage{color}
\usepackage{url}
\usepackage{enumerate}
\usepackage{geometry}
\usepackage{enumerate}
\usepackage{hyperref}
   
\usepackage{graphicx}

\newtheorem{thm}{Theorem}[section]
\newtheorem{lem}[thm]{Lemma}
\newtheorem{prop}[thm]{Proposition}

\newtheorem{rmk}[thm]{Remark}
\newtheorem{hyp}[thm]{Hypothesis}

\newtheorem{theo}{Theorem}

\numberwithin{equation}{section}

\title{Poisson boundary of a relativistic diffusion \\ in curved space-times: an example}

\author{J\"urgen Angst \footnote{Address: IRMAR, Universit\'e Rennes 1, Campus de Beaulieu, 35042 Rennes Cedex, France, \hspace{3cm}email: \url{jurgen.angst@univ-rennes1.fr}}}

\begin{document}
\maketitle

\begin{abstract} We study in details the long-time asymptotic behavior of a relativistic diffusion taking values in the unitary tangent bundle of a curved Lorentzian manifold, namely a spatially flat and fast expanding Robertson-Walker space-time. We prove in particular that the Poisson boundary of the diffusion can be identified with the causal boundary of the underlying manifold.\end{abstract}
%
%

\section{Introduction}
Considering the importance of heat kernels in Riemannian geometry, it appears very natural to investigate the links between geometry and asymptotics of Brownian paths in a Lorentzian setting. 
Extending Dudley's seminal work \cite{dudley1}, J. Franchi and Y. Le Jan constructed in \cite{flj}, on the unitary tangent bundle $T^1 \mathcal M$ of an arbitrary Lorentz manifold $\mathcal M$, a diffusion process which is Lorentz-covariant. This process, that we will simply call \textit{relativistic diffusion} in the sequel, is the Lorentzian
analogue of the classical Brownian motion on a Riemannian manifold. It can be seen either as a random perturbation of the timelike geodesic flow on the unitary tangent bundle, or as a stochastic development of Dudley's diffusion in a fixed tangent space. 
\par 
\medskip
In Minkowski space-time, and more generally in Lorentz manifolds of constant curvature, the long-time asymptotics of the relativistic diffusion is well understood, see \cite{dudley1,dudley3, flj,ismael,camille}. But, as in the Riemannian case, there is no hope to fully determinate the asymptotic behavior of the relativistic diffusion on an arbitrary Lorentzian manifold: it could depend heavily on the base space, see e.g. \cite{marcanton} and its references in the ``simple'' case of Cartan--Hadamard manifolds.
Recently in \cite{angstannIHP}, we studied in details the long-time asymptotic behavior of the relativistic diffusion in the case when the underlying space-time belong to a large class of curved Lorentz manifolds: Robertson-Walker space-times, or RW space-times for short, whose definition is recalled in Sect. \ref{sec.RW} below. We show in particular that the relativistic diffusion's paths converge almost surely to random points of a natural geometric compactification of the base manifold $\mathcal M$, namely its causal boundary $\partial \mathcal M_c^+$ introduced in the reference \cite{geroch}.

\newpage
\begin{theo}[Theorem 3.1 in \cite{angstannIHP}] 
Let $\mathcal M$ be a RW space-time,  $(\xi_0, \dot{\xi}_0) \in T^1 \mathcal M$ and let $(\xi_s, \dot{\xi}_s)_{0 \leq s \leq \tau}$ be the relativistic diffusion starting from $(\xi_0, \dot{\xi_0})$. Then, almost surely as $s$ goes to the explosion time $\tau$ of the diffusion, the first projection $\xi_s$ converges to a random point $\xi_{\infty}$ of the causal boundary $\partial \mathcal M_c^+$ of $\mathcal M$.
\end{theo}

The purpose of this paper is to push the analysis further by showing that, in the case of a RW space-time  with exponential growth, the Poisson boundary of the diffusion is precisely generated by the single random variable $\xi_{\infty}$ of the causal boundary $\partial \mathcal M_c^+$, which in that case can be identified with a spacelike copy of the Euclidean space $\mathbb R^3$ (see \cite{angstannIHP} and Theorem 4.3 of \cite{flores}). Namely, we have:

\begin{theo}[Theorems \ref{the.asymp} and \ref{the.poisson} below]
Let $\mathcal M:=(0, +\infty) \times_{\alpha} \mathbb R^3$ be a RW space-time where $\alpha$ has exponential growth. Let  $(\xi_s, \dot{\xi}_s)_{s \geq 0}=(t_s,x_s, \dot{t}_s, \dot{x}_s)_{s \geq 0}$ be the relativistic diffusion starting from $(\xi_0, \dot{\xi}_0) \in T^1 \mathcal M$. Then, almost surely as $s$ goes to infinity, the process $x_s$ converges to a random point $x_{\infty}$ in $\mathbb R^3$, and the invariant sigma field of the whole diffusion $(\xi_s, \dot{\xi}_s)_{s \geq 0}$ coincides almost surely with $\sigma(x_{\infty})$.
\end{theo}

The above result is the first computation of the Poisson boundary of the relativistic diffusion in the case of a Lorentz manifold with non-constant curvature. It can be seen as a complementary result of those of \cite{ismael, camille} in the flat cases. The difficulty here lies in the following facts: classical coupling techniques are hardly implemented in hypoelliptic settings, classical Lie group methods or explicit conditionning (Doob transform) do not apply in the presence of curvature. Our approach is purely probabilistic, it is first based on the existence of natural subdiffusions due to symmetries of the base manifold, namely the processes 
$(t_s,\dot{t}_s)_{s \geq 0}$ and $(t_s, \dot{t}_s, \dot{x}_s/|\dot{x}_s|)_{s \geq 0}$ are subdiffusions of the whole process. Using successive shift-couplings, we then show that the Poisson boundaries of these two subdiffusions are trivial. Finally, we conclude via an abstract conditionning argument, allowing us to show that 
the invariant sigma field of the whole diffusion is indeed generated by the invariant sigma field of the subdiffusion $(t_s, \dot{t}_s, \dot{x}_s/|\dot{x}_s|)_{s \geq 0}$
and the single extra information $\sigma(x_{\infty})$.

\par
\medskip
This last argument is new and non-trivial, it takes into account the hypoellipticity of  the infinitesimal generator of the relativistic diffusion and it's equivariance with respect to Euclidean spatial translations of the base space. It is actually the starting point of the very recent work \cite{AT1} in collaboration with C. Tardif, where our main motivation is to exhibit a general setting and some natural conditions that allow to compute the Poisson boundary of a diffusion starting from the Poisson boundary of a subdiffusion. It is interesting to note that the resulting devissage method can be used to recover Bailleul's result \cite{ismael} in a more direct way, see Section 4.2 of \cite{AT1}.

\par
\medskip
The article is organized as follows. In the first section, we briefly recall the geometrical background on RW space-times and the definition of the relativistic diffusion in this setting. In Section \ref{sec.results}, we then state our results concerning the asymptotic behavior of the relativistic diffusion and its Poisson boundary. The last section is dedicated to the proofs of these results. For the sake of self-containedness and in order to provide an easily readable article, some results of \cite{angstannIHP} and their proofs are recalled here.

\newpage
\section{Geometric and probabilistic backgrounds}
\label{sec.RW}\label{sec.background}
The Lorentz manifolds we consider here are RW space-times. They are the natural geometric framework to formulate the theory of Big-Bang in General Relativity theory.
The constraint that a space-time satisfies both Einstein's equations and the cosmological principle implies it has a warped product structure, see e.g. \cite{Weinberg} p. 395--404. A RW space-time, classically denoted by $\mathcal M:=I \times_{\alpha} M$, is thus defined as a Cartesian product of a open real interval $(I,-dt^2)$ (the base) and a 3-dimensional Riemannian manifold $(M, h)$ of constant curvature (the fiber), endowed with a Lorentz metric of the following form $g := -dt^2 + \alpha^2(t)h,$
where $\alpha$ is a positive function on $I$, called the \textit{expansion function}. Classical examples of RW space-times are the (half)$-$Minkowski space-time, Einstein static universe, de Sitter and anti-de Sitter space-times etc.
\par
\medskip
A detailed study of the relativistic diffusion in a general RW space-time has been led in \cite{angstannIHP} where we characterized the almost-sure long-time behavior of the diffusion. We focus here on the case when the real interval $I$ is unbounded and the Riemannian fiber is Euclidean, see Remark \ref{rem.courb} below. 
Namely, we consider RW space-times $\mathcal M=(0, +\infty) \times_{\alpha} \mathbb R^3$, where the expansion function $\alpha$ satisfies the following 
hypotheses: 

\begin{hyp}\noindent
\begin{enumerate}
\item The function $\alpha$ is of class $C^2$ on $(0, +\infty)$ and it is increasing and $\mathrm{log}-$concave, \emph{i.e.} the Hubble function $H:=\alpha'/\alpha$ is non-negative and non-increasing.
\item The function $\alpha$ has exponential growth, \emph{i.e.} the limit $H_{\infty}:=\displaystyle{\lim_{t \to + \infty} H(t)}$ is positive.
\end{enumerate}
\end{hyp}

\begin{rmk}\label{rem.courb}
The hypothesis of log-concavity of the expansion function is classical, it appears natural from both physical \cite{hawell} and mathematical points of view \cite{alias}. Note also that we are working in dimension $3+1$ because physically relevent space-times have dimension 4, but our results apply verbatim in dimension $d+1$ if $d \geq 3$. Finally, as noticed in Remark 3.5 of \cite{angstannIHP}, in a RW space-time $\mathcal M:=I \times_{\alpha} M$, if the expansion is exponential (the inverse of $\alpha$ is thus integrable at infinity), whatever the curvature of the Riemannian manifold $M$, the process $\dot{x}_s/|\dot{x}_s|$ asymptotically describes a recurrent time-changed spherical Brownian motion in the limit unitary tangent space, i.e. it does not ``see" the curvature $M$, that is why we concentrate here on the case $M=\mathbb R^3$. 
\end{rmk}

\begin{rmk}
A RW space-time  $\mathcal M=(0, +\infty) \times_{\alpha} \mathbb R^3$ is naturally endowed with a global chart $\xi=(t,x)$ where $x=(x^1, x^2,x^3)$ are the canonical coordinates in $\mathbb R^3$. At a point $(t,x)$, the scalar curvature is given by $R= -6 ( \alpha''(t)/\alpha(t)+ {\alpha'}^2(t)/\alpha^2(t))$. In particular, although spatially flat, such a space-time is not globally flat in general. In the case of a ``true'' exponential expansion, i.e. when $\alpha(t)=\exp(H \times t)$ for a positive constant $H$, the space-time $\mathcal M=(0, +\infty) \times_{\alpha} \mathbb R^3$ is an Einstein manifold: its Ricci tensor is proportional to its metric.  
\end{rmk}

On a general Lorentzian manifold $\mathcal M$, the sample paths $(\xi_s, \dot{\xi}_s)$ of the relativistic diffusion introduced in \cite{flj} are time-like curves that are future-directed and parametrized by the arc length $s$ so that the diffusion actually lives on the positive part of the unitary tangent bundle of the manifold, that we denote by $T^1_+ \mathcal M$.
The infinitesimal generator of the diffusion is the following hypoelliptic operator 
\[
\mathcal L:= \mathcal L_0 + \frac{\sigma^2}{2 } \Delta_{\mathcal V},
\]
where $\mathcal L_0$ generates the geodesic flow on $T^1 \mathcal M$, $\Delta_{\mathcal V}$ is the vertical Laplacian, and $\sigma$ is a real parameter. Equivalently, if $\xi^{\mu}$ is a local chart on $\mathcal M$ and if $\Gamma_{\nu \rho}^{\mu}$ are the usual Christoffel symbols, the relativistic diffusion is the solution of the following system of stochastic differential equations, for $0 \leq \mu \leq d=\mathrm{dim} (\mathcal M)$:
\begin{equation}\label{eqn.flj}
\left \lbrace \begin{array}{l}
\displaystyle{ d \xi^{\mu}_s = \dot{\xi}_s^{\mu} ds}, \\
\displaystyle{ d\xi^{\mu}_s= -\Gamma_{\nu \rho}^{\mu}(\xi_s)\, \xi^{\nu}_s \xi^{\rho}_s ds + d \times \frac{\sigma^2}{2}\, \xi^{\mu}_s ds+ \sigma  d M^{\mu}_s}, 
\end{array}\right.
\end{equation}
where the brakets of the martingales $M^{\mu}_s$ are given by
\[
\langle dM_s^{\mu}, \;  dM_s^{\nu}\rangle = (\xi^{\mu}_s \xi^{\nu}_s +g^{\mu \nu}(\xi_s))ds.
\]

In the case of a RW space-time of the form $\mathcal M=(0, +\infty) \times_{\alpha} \mathbb R^3$ endowed with its natural global chart, the metric is $g_{\mu \nu} = \mathrm{diag}(-1, \alpha^2(t), \alpha^2(t), \alpha^2(t))$, and the only non vanishing Christoffel symbols are $\Gamma_{i\,i}^0 = \alpha(t) \alpha'(t)$, and $\Gamma_{0\, i}^i = H(t)$ for $i=1, 2, 3$. 
Thus, in the case of a spatially flat RW space-time, the system of stochastic differential equations (\ref{eqn.flj}) that defines the relativistic diffusion simply reads: 
\begin{equation}\label{eqn.flj.eucli}
\left \lbrace \begin{array}{ll}
\displaystyle{d t_s=\dot{t}_s ds}, & \quad \displaystyle{d \dot{t}_s = - \alpha(t_s) \: \alpha'(t_s) |\dot{x}_s|^2 ds + \frac{3 \sigma^2}{2} \dot{t}_s ds + d M^{\dot{t}}_s,} \\
\displaystyle{d x^i_s = \dot{x}_s^i ds}, &  \quad \displaystyle{d \dot{x}^i_s = \left(-2 H(t_s) \dot{t}_s + \frac{3 \sigma^2}{2}\right)  \dot{x}^i_s\, ds + dM^{\dot{x}^i}_s},
\end{array}\right. 
\end{equation}
where $|\dot{x}_s|$ denote the usual Euclidean norm of $\dot{x}_s$ in $\mathbb R^3$ and the brackets satisfy
\[
\left \lbrace \begin{array}{l}
\displaystyle{d \langle M^{\dot{t}}, \, M^{\dot{t}} \rangle_s = \sigma^2 \left( \dot{t}_s^2-1 \right) ds, \quad d \langle M^{\dot{t}}, \, M^{\dot{x}^i} \rangle_s = \sigma^2 \, \dot{t}_s \dot{x}^i_s ds,} \\
\displaystyle{d \langle M^{\dot{x}^i}, \, M^{\dot{x}^j} \rangle_s = \sigma^2 \left(\dot{x}^i_s \dot{x}^j_s + \frac{\delta_{ij}}{\alpha^2(t_s)}  \right) ds}.
\end{array}\right.
\]
Moreover, the parameter $s$ being the arc length, we have the pseudo-norm relation:
\begin{equation} \label{eqn.pseudo.eucli}\dot{t}^2_s -1 = \alpha^2(t_s) \times |\dot{x}_s|^2.\end{equation}

\begin{rmk}\label{rem.pseudo}
The sample paths being future-directed, from the above pseudo-norm relation, we have obviously $\dot{t}_s \geq 1$, in particular as long as it is well defined, the ``time" process $t_s$ is a strictly increasing and we have $t_s>s$. 
\end{rmk}

\section{Statement of the results}\label{sec.results}
We can now state our results concerning the asymptotic behavior of the relativistic diffusion and its Poisson boundary in a spatially flat and fast expanding RW space-time. 
For the sake of clarity, the proofs of these different results are postponed in Section \ref{sec.proofs}. For the whole section, let us thus fix a spatially flat RW space-time $\mathcal M=(0, +\infty) \times_{\alpha} \mathbb R^3$, where $\alpha$ satisfies the hypotheses stated in Sect. \ref{sec.RW}.

\subsection{Existence, uniqueness, reduction of the dimension}
Naturally, the first thing to do is to ensure that the system of stochastic differential equations (\ref{eqn.flj.eucli}) admits a solution, and possibly to exhibit lower dimensional subdiffusions that will facilitate its study. This is the object of the following proposition.
 
\begin{prop}\label{pro.exiuni}
For any $(\xi_0, \dot{\xi}_0)=(t_0, x_0, \dot{t}_0, \dot{x}_0) \in T^1_+ \mathcal M$, the system of stochastic differential equations (\ref{eqn.flj.eucli})
admits a unique strong solution $(\xi_s, \dot{\xi}_s)=(t_s, x_s, \dot{t}_s, \dot{x}_s)$ starting from $(\xi_0, \dot{\xi}_0)$, which is well defined for all positive proper times $s$. Moreover, this solution admits the two following subdiffusions of dimension two and four respectively: $(t_s, \dot{t}_s)_{s \geq 0}$ and $(t_s, \dot{t}_s, \dot{x}_s/|\dot{x}_s| )_{s \geq 0}$. 
\end{prop}
\begin{rmk}
Given a point $(\xi, \dot{\xi})\in T^1_+ \mathcal M$, we will denote by $\mathbb P_{(\xi, \dot{\xi})}$ the law of the relativistic diffusion starting from $(\xi, \dot{\xi})$ and by $\mathbb E_{(\xi, \dot{\xi})}$ the associated expectation. Unless otherwise stated, the word ``almost surely'' will mean $\mathbb P_{(\xi, \dot{\xi})}-$almost surely. The two above subdiffusions will be called the \emph{temporal} and \emph{spherical} diffusions respectively.
\end{rmk}


\subsection{Asymptotics of the relativistic diffusion}\label{sec.asymp}
As conjectured in \cite{flj}, we show that the relativistic diffusion asymptotically behaves like light rays, \emph{i.e.} light-like geodesics. Indeed, from Remark \ref{rem.pseudo}, we know that the first projection $t_s$ of the (non-Markovian) process $\xi_s=(t_s, x_s)  \in \mathcal M$ goes almost-surely to infinity with $s$. We shall prove that its spatial part $x_s$ converges almost surely to a random point $x_{\infty}$ in $\mathbb R^3$, so that the diffusion asymptotically follows a line $D_{\infty}$ in $\mathcal M$, see Fig. \ref{fig.asymp} below, which is the typical behavior of a light-like geodesic. Moreover, we shall see that the normalized derivative $\dot{x}_s/|\dot{x}_s|$ is recurrent so that the curve $(\xi_s)_{s \geq 0}$ actually winds along the line $D_{\infty}$ in a recurrent way.

\begin{figure}[ht]
\hspace{2.5cm}\scalebox{0.7}{\begin{picture}(0,0)%
\includegraphics{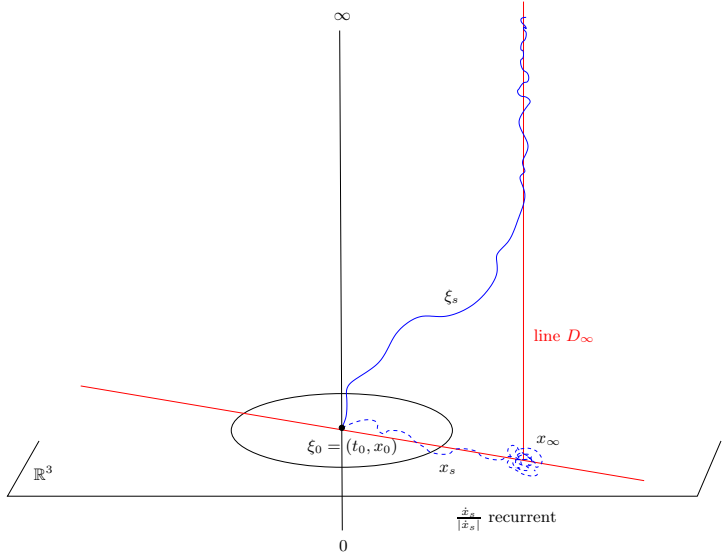}%
\end{picture}}%
\setlength{\unitlength}{2072sp}%
\begingroup\makeatletter\ifx\SetFigFontNFSS\undefined%
\gdef\SetFigFontNFSS#1#2#3#4#5{%
  \reset@font\fontsize{#1}{#2pt}%
  \fontfamily{#3}\fontseries{#4}\fontshape{#5}%
  \selectfont}%
\fi\endgroup%
\scalebox{0.7}{\begin{picture}(12219,9523)(124,-8617)
\put(9181,-6676){\makebox(0,0)[lb]{\smash{{\SetFigFontNFSS{9}{10.8}{\familydefault}{\mddefault}{\updefault}{\color[rgb]{0,0,0}$x_{\infty}$}%
}}}}
\put(7606,-4246){\makebox(0,0)[lb]{\smash{{\SetFigFontNFSS{9}{10.8}{\familydefault}{\mddefault}{\updefault}{\color[rgb]{0,0,0}$\xi_s$}%
}}}}
\put(9136,-4921){\makebox(0,0)[lb]{\smash{{\SetFigFontNFSS{9}{10.8}{\familydefault}{\mddefault}{\updefault}{\color[rgb]{1,0,0}line $D_{\infty}$}%
}}}}
\put(586,-7306){\makebox(0,0)[lb]{\smash{{\SetFigFontNFSS{9}{10.8}{\familydefault}{\mddefault}{\updefault}{\color[rgb]{0,0,0}$\mathbb R^3$}%
}}}}
\put(5716,569){\makebox(0,0)[lb]{\smash{{\SetFigFontNFSS{9}{10.8}{\familydefault}{\mddefault}{\updefault}{\color[rgb]{0,0,0}$\infty$}%
}}}}
\put(5806,-8521){\makebox(0,0)[lb]{\smash{{\SetFigFontNFSS{9}{10.8}{\familydefault}{\mddefault}{\updefault}{\color[rgb]{0,0,0}$0$}%
}}}}
\put(7516,-7171){\makebox(0,0)[lb]{\smash{{\SetFigFontNFSS{9}{10.8}{\familydefault}{\mddefault}{\updefault}{\color[rgb]{0,0,0}$x_s$}%
}}}}
\put(5266,-6811){\makebox(0,0)[lb]{\smash{{\SetFigFontNFSS{9}{10.8}{\familydefault}{\mddefault}{\updefault}{\color[rgb]{0,0,0}$\xi_0=(t_0,x_0)$}%
}}}}
\put(7786,-8026){\makebox(0,0)[lb]{\smash{{\SetFigFontNFSS{9}{10.8}{\familydefault}{\mddefault}{\updefault}{\color[rgb]{0,0,0}$\frac{\dot{x}_s}{|\dot{x}_s|}$ recurrent}%
}}}}
\end{picture}}%
\caption{Typical path of the relativistic diffusion in $\mathcal M=(0, +\infty) \times_{\alpha} \mathbb R^3$.}
\label{fig.asymp}
\end{figure}

To state precise results, let us introduce the following notations.
Given two positive constants $a$ and $b$, let $\nu_{a,b}$ be the probability measure on $(1,+\infty)$ admitting the following density with respect to Lebesgue measure:
\[
\nu_{a, b}(x) :=C_{a,b} \times \sqrt{x^2-1} \times \exp \left(- \frac{2a}{b^2} x \right),
\]
where $C_{a,b}$ is the normalizing constant. If $f \in \mathbb L^1(\nu_{a, b})$, we will write 
$\nu_{a, b}(f):=\int f(x) \nu_{a, b}(x)dx.$
The following theorem summarizes the almost sure asymptotics of the relativistic diffusion, its proofs is given in Sect. \ref{sec.proofasymp} below.

\begin{thm}\label{the.asymp}
Let $(\xi_0, \dot{\xi}_0) \in T^1_+ \mathcal M$, and let $(\xi_s, \dot{\xi}_s)=(t_s, x_s, \dot{t}_s, \dot{x}_s)$ be the relativistic diffusion starting from $(\xi_0, \dot{\xi}_0)$. Then as $s$ goes to infinity, we have the following almost sure asymptotics:
\begin{enumerate}
\item the  non-Markovian process $\dot{t}_s$ is Harris-recurrent in $(1, +\infty)$. Moreover, if $f$ is a monotone, $\nu_{H_{\infty},\sigma}-$integrable function, or if it is bounded and continuous: 
\[
\lim_{s \to +\infty} \frac{1}{s} \: \int_0^s f(\dot{t}_u)du \stackrel{a.s.}{=} \nu_{H_{\infty},\sigma}(f). 
\]
In particular, $\displaystyle{\lim_{s \to +\infty} t_s /s  \stackrel{a.s.}{=} \nu_{H_{\infty},\sigma}(\mathrm{Id}) \in (0,+\infty)}$.
\item the spatial projection $x_s$ converges almost surely to a random point $x_{\infty}$ in $\mathbb R^3$.
\item the normalized spatial derivative $\dot{x}_s /|\dot{x}_s|$ is a time-changed Brownian motion on the sphere $\mathbb S^2 \subset \mathbb R^3$, in particular it is recurrent.
\end{enumerate}
\end{thm}

As noticed in the introduction, the above asymptotic results can be rephrased concisely thanks to the notion of causal boundary introduced in \cite{geroch}. In fact, in a fast expanding RW space-time $\mathcal M=(0, +\infty) \times_{\alpha} \mathbb R^3$, the causal boundary $\partial \mathcal M_c^+$ identifies with a spacelike copy of $\mathbb R^3$: a causal curve $\xi_u=(t_u, x_u)$ converges to a point $\xi_{\infty} \in \partial \mathcal M_c^+$ iff $t_u \to +\infty$ and $x_u \to x_{\infty} \in \mathbb R^3$, see \cite{flores}.

\subsection{Poisson boundary of the relativistic diffusion}
We now describe the Poisson boundary of the relativistic diffusion, that is we determine its invariant sigma field $\textrm{Inv}((\xi_s, \dot{\xi}_s)_{s \geq 0})$ or equivalently the set of bounded harmonic functions with respect to its infinitesimal generator $\mathcal L$. Owing to Theorem \ref{the.asymp}, the processes $\dot{t}_s$ and $\dot{x}_s/|\dot{x}_s|$ being recurrent, it is tempting to assert that the only non trivial asymptotic variable associated to the relativistic diffusion is the random point $x_{\infty} \in \mathbb R^3$. Nevertheless, 
the result is far from being trivial because some extra-information relating the temporal components and the spatial ones could be hidden. For example, in Minkowski space-time i.e. if $\alpha \equiv 1$, we have almost surely $t_s \to +\infty$ and $x_s/|x_s| \to \theta_{\infty} \in \mathbb S^2$, but the Poisson boundary of the relativistic diffusion is not reduced to the sigma algebra $\sigma(\theta_{\infty})$ because the difference $t_s - \langle x_s, \theta_{\infty}\rangle$ converges almost-surely to a random variable that is not measurable with respect to $\sigma(\theta_{\infty})$. 
\par
\smallskip
Using shift-coupling techniques, we first prove in Sect. \ref{sec.proofpoisson} (Propositions \ref{pro.poissontemp} and \ref{pro.poissonspatial}) that the invariant sigma fields of the temporal and spherical subdiffusions are indeed trivial. Then, taking into account the regularity of harmonic functions (hypoellipticity), and using the equivariance of infinitesimal generator under Euclidean translations, we construct an abstract conditionning  (Proposition \ref{pro.poisson}) to end up with the following result:

\begin{thm}\label{the.poisson}
Let $\mathcal M:=(0, +\infty) \times_{\alpha} \mathbb R^3$ be a RW space-time where $\alpha$ has exponential growth. Let $(\xi_0, \dot{\xi}_0) \in T^1_+ \mathcal M$ and let $(\xi_s, \dot{\xi}_s)=(t_s, x_s, \dot{t}_s, \dot{x}_s)$ be the relativistic diffusion starting from $(\xi_0, \dot{\xi}_0)$. Then, the invariant sigma field $\textrm{Inv}((\xi_s, \dot{\xi}_s)_{s \geq 0})$ of the whole diffusion coincides with the sigma field generated by the single variable $x_{\infty} \in \mathbb R^3$ up to $\mathbb P_{(\xi_0, \dot{\xi}_0)}-$negligeable sets. Equivalently, if $h$ is a bounded $\mathcal L-$harmonic function on $T^1_+ \mathcal M$, there exists a bounded measurable function $\psi$ on $\mathbb R^3$, such that
\[
 h(\xi, \dot{\xi}) = \mathbb E_{(\xi, \dot{\xi})}[\psi(x_{\infty})], \; \; \forall (\xi, \dot{\xi}) \in  T^1_+ \mathcal M.
 \]
\end{thm}

In other words, all the asymptotic information on the relativistic diffusion is encoded in the point $x_{\infty} \in \mathbb R^3$ or equivalently in the point $\xi_{\infty} =(\infty, x_{\infty})$ of the causal boundary $\partial \mathcal M_c^+$. The above theorem is thus very similar to Theorem 1 of \cite{ismael} asserting that the invariant sigma field of the relativistic diffusion in Minkowski space-time is generated by a random point on its causal boundary, which in that case identifies with the product $\mathbb R^+ \times \mathbb S^2$. It is tempting to ask if such a link between the Poisson and causal boundary holds in a more general context. This question is the object of a work in progress of the author and C. Tardif \cite{AT2}.

\section{Proofs of the results}\label{sec.proofs}
This last section is dedicated to the proofs of the different results stated above. Namely, the section \ref{sec.proofuniexi} below is devoted to the proof of Proposition \ref{pro.exiuni}, and in Sections \ref{sec.proofasymp} and \ref{sec.proofpoisson} we give the proofs of Theorems \ref{the.asymp} and \ref{the.poisson} respectively.

\subsection{Existence, uniqueness, reduction of the dimension}\label{sec.proofuniexi}
We first give the proof of Proposition \ref{pro.exiuni}  concerning the existence, the uniqueness and the lifetime of the relativistic diffusion. The coefficients in the system of stochastic differential equations (\ref{eqn.flj.eucli}) being smooth, the first assertions follow from classical existence and uniqueness theorems, see for example Theorem (2.3) p. 173 of \cite{ikeda}. Next, the fact that the temporal process $(t_s, \dot{t}_s)$ is a subdiffusion of the whole relativistic diffusion is an immediate consequence of Equation (\ref{eqn.flj.eucli}) and the pseudo-norm relation (\ref{eqn.pseudo.eucli}), which allows to express the norm of the spatial derivative $\dot{x}_s$ in term of the temporal process. Finally, the analogous result concerning the spherical subdiffusion follows from a straightforward computation, namely setting $\Theta_s=(\Theta_s^1, \Theta_s^2, \Theta_s^3)$ where $\Theta^i_s:=\dot{x}_s^i/|\dot{x}_s|$ to lighten the expressions, we have the following lemma:

\begin{lem}\label{cor.eucli.cartesien}
The temporal process $(t_s, \dot{t}_s)$ and the spherical process $(t_s, \dot{t}_s, \Theta_s)$ are solutions of the following system of stochastic differential equations: 
\begin{eqnarray} 
         \displaystyle{d t_s=\dot{t}_s ds,} \quad     \displaystyle{d \dot{t}_s = -H(t_s) (\dot{t}_s^2-1)ds+ \frac{3 \sigma^2}{2} \dot{t}_s ds + d M^{\dot{t}}_s,}
          \label{eqn.ttpoint} \\
     \displaystyle{d \Theta^i_s = - \frac{\sigma^2}{\dot{t}_s^2-1} \times  \Theta^i_s ds + d M^{\Theta^i}_s,} 
        \label{eqn.sphere} 
\end{eqnarray}
where the brakets of the martingales $M^{\dot{t}}$ and $M^{\Theta^i}$ are given by
\begin{equation}\label{eqn.crochet}
\left \lbrace \begin{array}{l}
\displaystyle{d \langle M^{\dot{t}}, \, M^{\dot{t}} \rangle_s = \sigma^2 \left( \dot{t}_s^2-1 \right) ds,} \quad
\displaystyle{d \langle M^{\dot{t}}, \, M^{\Theta^i} \rangle_s = 0,} \\
\\
\displaystyle{d \langle M^{\Theta^i},  M^{\Theta^j} \rangle_s = \frac{\sigma^2}{\dot{t}_s^2-1}  \left(\delta_{ij} -  \Theta^i_s \Theta^j_s \right)ds}.
\end{array}\right.\end{equation}
\end{lem}

\begin{rmk}From Remark \ref{rem.pseudo}, we know that $\dot{t}_s \geq 1$ a.s. for all $s \geq 0$. In fact, the Hubble fuction $H=\alpha'/\alpha$ being non-increasing, using standard comparison techniques, it is easy to see that $\dot{t}_s>1$ a.s for all $s>0$, so that the term $\dot{t}_s^2-1$ in the denominators above never vanishes.
\end{rmk}

\begin{rmk}\label{rem.brownienspherique}
The process $\Theta_s$ is a time-changed spherical Brownian motion. More precisely, introducing the clock 
\[
C_s := \sigma^2 \int_0^s \frac{du}{\dot{t}_u^2-1} du,
\]
 the time-changed process $\widetilde{\Theta}$ such that $ \widetilde{\Theta}(C_s) := \Theta_s$ is a standard spherical Brownian motion on $\mathbb S^2 \subset \mathbb R^3$ and it is independent of the temporal subdiffusion. 
\end{rmk}

\subsection{Asymptotic behavior of the diffusion}\label{sec.proofasymp}
We now prove the results stated in Theorem \ref{the.asymp} concerning the asymptotic behavior of the relativistic diffusion.
We distinguish the cases of the temporal components of the diffusion (Proposition \ref{pro.temp} below) and its spatial components (Proposition \ref{pro.spatial}). 

\subsubsection{Asymptotic behavior of the temporal subdiffusion}
In this paragraph, the word ``almost sure'' refers to the law of the temporal subdiffusion solution of Equation (\ref{eqn.ttpoint}). The first point of Theorem \ref{the.asymp}
corresponds to the following proposition. 
\begin{prop}\label{pro.temp}
Let $(t_0, \dot{t}_0) \in (0, +\infty) \times [1, +\infty)$ and let $(t_s, \dot{t}_s)$ be the solution of Equation (\ref{eqn.ttpoint}) starting from $(t_0, \dot{t}_0)$.
Then, the process $\dot{t}_s$ is Harris-recurrent in $(1, +\infty)$ and if $f$ is a monotone, $\nu_{H_{\infty},\sigma}-$integrable function, or if it is bounded and continuous: 
\[ 
\lim_{s \to +\infty} \frac{1}{s} \: \int_0^s f(\dot{t}_u)du \stackrel{a.s.}{=} \nu_{H_{\infty},\sigma}(f). 
\]
In particular 
\[ 
\lim_{s \to +\infty} \frac{t_s}{s}  =\lim_{s \to +\infty} \frac{1}{s} \: \int_0^s \dot{t}_u du \stackrel{a.s.}{=} \nu_{H_{\infty},\sigma}(\textrm{Id}). 
\]
\end{prop}

The proof of the proposition, which is given below, is based on standard comparison techniques and on the two following elementary lemmas. Recall that the Hubble function $H$ is supposed to be non-increasing.

\begin{lem}Given a constant $H_0>0$, $\dot{t}_0 \in [1, +\infty)$ and a real standard Brownian motion $B$, the following stochastic differential equation 
\[
d \dot{t}_s= -H_0 \times \left( \dot{t}_s^2-1 \right)ds  + \frac{3 \sigma^2}{2} \dot{t}_s ds+ \sigma \sqrt{\dot{t}_s^2-1} \,dB_s
\]
has a unique strong solution starting from $\dot{t}_0$, well defined for all times $s \geq 0$. Moreover, $\dot{t}_s$ admits the probability measure $\nu_{H_0,\sigma}$ introduced in Sect. \ref{sec.asymp} as an invariant measure. In particular, it is ergodic. 
\label{lem.ergo}\end{lem}

\begin{lem} \label{LEM.COMPARPROJ}
Let $(t_0, \dot{t}_0) \in (0, +\infty) \times [1, +\infty)$ and let $(t_s, \dot{t}_s)$ be the solution of Equation (\ref{eqn.ttpoint}) starting from $(t_0, \dot{t}_0)$, where the martingale $M^{\dot{t}}$ is represented by a real standard Brownian motion $B$, i.e. $d M^{\dot{t}}_s = \sigma ( \dot{t}_s^2-1)^{1/2} d B_s$. Let $u_s$ and $v_s$ be the unique strong solutions, well defined for all $s\geq 0$, and starting from $u_0=v_0=\dot{t}_0$, of the equations: 
\[
\begin{array}{l} 
\displaystyle{d u_s=  -H(t_0)\left(u_s^2-1 \right)ds + \frac{3 \sigma^2}{2} u_s ds+ \sigma \sqrt{u_s^2-1} dB_s}, \\
\displaystyle{d v_s=  -H_{\infty}\left(v_s^2-1 \right)ds + \frac{3 \sigma^2}{2} v_s ds+ \sigma \sqrt{v_s^2-1} dB_s}.   
\end{array}
\]
Then, almost surely, for all $0 \leq s <+\infty$, one has $\displaystyle{ u_s \leq \dot{t}_s \leq v_s}$.
\end{lem}

\begin{proof}[Proof of Proposition \ref{pro.temp}]
There exists a standard Brownian motion $B$ such that the temporal process $(t_s, \dot{t}_s)$ is the solution of the stochastic differential equations
\[
d t_s = \dot{t}_s ds, \quad d \dot{t}_s= -H(t_s) \times \left( \dot{t}_s^2-1 \right)ds  + \frac{3 \sigma^2}{2} \dot{t}_s ds+ \sigma \sqrt{\dot{t}_s^2-1} \,dB_s.
\]
Let $z_s$ be the unique strong solution, starting from $z_0=\dot{t}_0$, of the stochastic differential equation
\[
d z_s= -H_{\infty} \times \left( |z_s|^2-1 \right)ds  + \frac{3 \sigma^2}{2} z_s ds+ \sigma \sqrt{|z_s|^2-1} \,dB_s.
\]
For $n \in \mathbb N$, let $z^n_s$ be the process that coincides with $\dot{t}_s$ on $[0,n]$ and is the solution on $[n, +\infty)$ of the stochastic differential equation
\[
d z^n_s= -H(t_0+n) \times \left( |z^n_s|^2-1 \right)ds  + \frac{3 \sigma^2}{2} z^n_s ds+ \sigma \sqrt{|z^n_s|^2-1} \,dB_s.
\]
By Lemma \ref{LEM.COMPARPROJ}, for all $n \geq 0$ and $s \geq 0$, one has $z^n_s \leq \dot{t}_s \leq z_s$.
By Lemma \ref{lem.ergo}, both processes $z^0_s$ and $z_s$ are ergodic in $(1, +\infty)$, in particular, they are Harris recurrent and so is $\dot{t}_s$.
Now consider an increasing and $\nu_{H_{\infty},\sigma}-$integrable function $f$, and fix an $\varepsilon>0$.
For all $n \in \mathbb N$, the function $f$ is also integrable against the measure $\nu_{H(t_0+n),\sigma}$ and by dominated convergence theorem, $\nu_{H(t_0+n),\sigma}(f) $ converges to $\nu_{H_{\infty},\sigma}(f)$ when $n$ goes to infinity. 
Choose $n$ large enough so that we have $ | \nu_{H(t_0+n),\sigma}(f) - \nu_{H_{\infty},\sigma}(f)| \leq \varepsilon$. As $z^n_s \leq \dot{t}_s \leq z_s$ for $s \geq 0$, one has almost surely:
\[
\displaystyle{\int_0^s f(z^n_u)du \leq \int_0^s f(\dot{t}_u)du \leq \int_0^s f(z_u)du}.
\]
The integer $n$ being fixed, by the ergodic theorem, we have that almost surely:
\[
\nu_{H_{\infty},\sigma}(f)- \varepsilon \leq \nu_{H(t_0+n),\sigma}(f) \leq \liminf_{s \to + \infty} \frac{1}{s} \int_0^s f(\dot{t}_u)du \leq
 \limsup_{s \to + \infty}\frac{1}{s} \; \int_0^s f(\dot{t}_u)du \leq \nu_{H_{\infty},\sigma}(f).
 \]
Letting $\varepsilon$ goes to zero, we get the desired result.  
As any smooth function can be written as the difference of two monotone functions, the above convergence extends to functions in the set $C^{1}_b=\{ f, \; f' \; \textrm{is bounded on} \; (1, +\infty)\}$, and then by regularization, to the set of bounded continuous functions on $(1,+\infty)$. 
\end{proof}

\subsubsection{Asymptotic behavior of the spatial components} 
The second and third points of Theorem \ref{the.asymp} are the object of the next proposition:
\begin{prop}  \label{pro.spatial}
Let $(\xi_s, \dot{\xi}_s)=(t_s, x_s, \dot{t}_s, \dot{x}_s)$ be the relativistic diffusion starting from $(\xi_0, \dot{\xi}_0) \in T^1_+ \mathcal M$. Then, as $s$ goes to infinity, the spatial projection $x_s$ converges almost surely to a random point $x_{\infty} \in \mathbb R^3$, and the process $\Theta_s=\dot{x}_s /|\dot{x}_s|$ is recurrent in $\mathbb S^2 \subset \mathbb R^3$.
\end{prop}

\begin{proof}
From Equation (\ref{eqn.pseudo.eucli}), we have $|\dot{x}_s|=\sqrt{\dot{t}_s^2-1}/\alpha(t_s)$ for all $s \geq 0$. Therefore
\[
|x_s -x_0| \leq \int_0^s |\dot{x}_u |du = \int_0^s \frac{\sqrt{\dot{t}_u^2-1}}{\alpha(t_u)} du \leq \int_0^s \frac{\dot{t}_u}{\alpha(t_u)} du =\int_{t_0}^{t_s} \frac{du}{\alpha(u)}.
\]
The process $t_s$ goes almost surely to infinity with $s$. The expansion function $\alpha$ having exponential growth, the last integral is a.s. convergent, so that the total variation of $x_s$ and the process itself are also convergent, whence the first point in the proposition. According to Remark \ref{rem.brownienspherique}, $\dot{x}_s /|\dot{x}_s|=\Theta_s=\widetilde{\Theta}_{C_s}$ is a time-changed spherical Brownian motion. By Proposition \ref{pro.temp}, we have  $\lim_{s \to +\infty} C_s /s \in (0, +\infty)$, in particular the clock $C_s$ goes almost surely to infinity with $s$ and the process $\Theta_s$ is recurrent in $\mathbb S^2$.
\end{proof}

\subsection{Poisson boundary of the relativistic diffusion} \label{sec.proofpoisson}
\noindent
The proof of Theorem \ref{the.poisson} is divided into three parts. We first prove a Liouville theorem for the temporal subdiffusion (Proposition \ref{pro.poissontemp}), then we prove an analogous result for the spherical subdiffusion (Proposition \ref{pro.poissonspatial}). Finally, we deduce the Poisson boundary of the global relativistic diffusion (Proposition \ref{pro.poisson}). 
\subsubsection{A Liouville theorem for the temporal subdiffusion}
\noindent
The infinitesimal generator of the temporal subdiffusion $(t_s, \dot{t}_s)$, acting on smooth functions from $(0,+\infty) \times [1, +\infty)$ to $\mathbb R$, is given by
$$
 \mathcal L_{H}:= \dot{t} \partial_{t}  - H(t) (\dot{t}^2-1) \partial_{\dot{t}}  + \frac{\sigma^2}{2} (\dot{t}^2-1) \partial_{\dot{t}}^2. 
$$
\begin{prop}
\label{pro.poissontemp}All bounded $\mathcal L_{H}-$harmonic functions are constant.
\end{prop}

\begin{proof}The proof of Proposition \ref{pro.poissontemp} is based on the following fact: there is an automatic shift coupling between two independent solutions of the system (\ref{eqn.ttpoint}). Let $B^1$ and $B^2$ be two independent standard Brownian motions defined on two measured spaces $(\Omega_1, \mathcal F_1)$ and $(\Omega_2, \mathcal F_2)$ as well as the processes $(t_s^1, \dot{t}_s^1)$ and $(t_s^2, \dot{t}_s^2)$, starting from $(t_0^1, \dot{t}_0^1) \neq (t_0^2, \dot{t}_0^2) $ (deterministic) and solution of the following systems, for $ i=1,2$:
\[
 dt_s^i= \dot{t}_s^i ds, \;\; d \dot{t}_s^i= \left[-H(t_s^i) \left(|\dot{t}_s^i|^2-1 \right) + \frac{3 \sigma^2}{2} \dot{t}_s^i \right]ds+ \sigma \sqrt{|\dot{t}_s^i|^2-1} d B^i_s.
 \]
Define $\tau_0:=\max(t_0^1,t_0^2)$. We denote by $\mathbb P_i$ the law of $(t_s^i, \dot{t}_s^i)$ and by $\mathbb P:=\mathbb P_1 \otimes \mathbb P_2$ the law of the couple. 
From Remark \ref{rem.pseudo}, the processes $t^i_s$ are strictly increasing. Denote by $(t^i)^{-1}_s$ their inverse, and define $u^i_s:= \dot{t}^i [ (t^i)^{-1}_s]$. Without loss of generality, one can suppose that $1<u_{\tau_0}^1< u_{\tau_0}^2$. By It\^o's formula, for $s \geq \tau_0$, one has  
\begin{equation}\label{eqn.logv}
\frac{1}{2} \log \left(\frac{|u_s^1|^2-1}{|u_s^2|^2-1} \right) =\frac{1}{2}  \log \left(\frac{|u_{\tau_0}^1|^2-1}{|u_{\tau_0}^2|^2-1} \right)  + Q_s + R_s+M_s,
\end{equation}
where 
\[
\begin{array}{ll}
\displaystyle{Q_s:=\sigma^2 \left[(t^1)^{-1}_s -(t^2)^{-1}_s\right]  - \sigma^2 \left[(t^1)^{-1}_{\tau_0} -(t^2)^{-1}_{\tau_0}\right], } \\
\\
\displaystyle{R_s:=\frac{\sigma^2}{2}\left(\int_{\tau_0}^{s}\frac{u_r^2 \left(|u_r^2|^2-1\right)-u_r^1 \left(|u_r^1|^2-1\right)}{u_r^1 \left(|u_r^1|^2-1\right) \times u_r^2 \left(|u_r^2|^2-1\right)}dr\right),}\\
\end{array}
\]
and where $M_s$ is a martingale whose bracket is given by: 
\begin{equation} \langle M \rangle_s = \sum_{i=1}^2 \int_{(t^i)^{-1}_{\tau_0}}^{(t^i)^{-1}_s}\frac{|\dot{t}_u^i|^2}{|\dot{t}_u^i|^2-1} du \geq (t^1)^{-1}_s- (t^1)^{-1}_{\tau_0}.
\label{eqn.minorcrochet}\end{equation}
Let us show that the coupling time $\tau_c := \inf \{ s > \tau_0, \; u^1_s=u^2_s\}$ is finite $\mathbb P-$almost surely. Consider the set 
$A:=\{ \omega \in \Omega_1 \times \Omega_2, \; \tau_c(\omega)=+\infty\}$. By definition, if $\omega\in A$ one has $u^1_s (\omega)< u^2_s(\omega)$ for $s > \tau_0$. We deduce that $R_s(\omega), Q_s(\omega)>0$ for all $s > \tau_0$. Indeed, for $s > \tau_0$, one has  : 
\[
 \int_{\tau_0}^s \frac{dr}{u_r^1} > \int_{\tau_0}^s \frac{dr}{u_r^2}, \;\; \hbox{and} \;\; \int_{\tau_0}^s \frac{dr}{u_r^i}= \int_{\tau_0}^s \frac{dv}{\dot{t}^i[(t^1)^{-1}_v]}=(t^i)^{-1}_s - (t^i)^{-1}_{\tau_0}.
 \]
On the set $A$, by Equation (\ref{eqn.logv}), the martingale $M_s$ thus admits the upper bound: 
\[
M_s +  \frac{1}{2} \log \left(\frac{|u_{\tau_0}^1|^2-1}{|u_{\tau_0}^2|^2-1} \right) \leq  \frac{1}{2} \log \left(\frac{|u_s^1|^2-1}{|u_s^2|^2-1} \right) \leq 0,
\]
But by Equation (\ref{eqn.minorcrochet}), as $(t^1)^{-1}_s$ goes to infinity with $s$, we have also $\langle M \rangle_{\infty} = +\infty$ $\mathbb P-$almost surely. Therefore $\mathbb P(A)=0$ and $\tau_c < +\infty \; \mathbb P-$almost surely. In other words, $\mathbb P-$a.s. the two sets $(t_.^1, \dot{t}_.^1)_{\mathbb R^+}$ and $(t_.^2, \dot{t}_.^2)_{\mathbb R^+}$ intersect, where $(t_.^i, \dot{t}_.^i)_{\mathbb R^+}$ denotes the set of points of the curves $(t_s^i, \dot{t}_s^i)_{s \geq 0}$, $i=1,2$. Let us define the random times
\[
\displaystyle{T_1:=\inf \{ s>0,  (t_s^1, \dot{t}_s^1) \in (t_.^2, \dot{t}_.^2)_{\mathbb R^+}} \},  \quad
\displaystyle{T_2:=\inf \{ s>0,  (t_s^2, \dot{t}_s^2) \in (t_.^1, \dot{t}_.^1)_{\mathbb R^+}} \}.
\]
These variables are not stopping times for the filtration 
$\sigma( (t_s^i, \dot{t}_s^i), \; i=1,2, \; s \leq u)_{u \geq 0}$, nevertheless they are finite $\mathbb P-$almost surely. As a consequence, we deduce that both sets
$A_1:= \{ \omega_1 \in \Omega_1, \; T_2 < +\infty \; \mathbb P_2-\hbox{a.s.} \}$ and
$A_2:= \{ \omega_2 \in \Omega_2, \; T_1 < +\infty \; \mathbb P_1-\hbox{a.s.} \}$
verify $\mathbb P_1(A_1)=\mathbb P_2(A_2)=1$. 
Moreover, as the processes $t_s^i$ are strictly increasing, one has 
\begin{equation} \label{eqn.couplage}(t_{T_1}^1, \dot{t}_{T_1}^1)= (t_{T_2}^2, \dot{t}_{T_2}^2) \quad \mathbb P-\hbox{almost surely}.\end{equation}
Indeed, by definition of $T_1$ and $T_2$, there exists $u, v \in \mathbb R^+$ (random) such that
$(t_{T_1}^1, \dot{t}_{T_1}^1)= (t_u^2, \dot{t}_u^2)$ and $(t_{T_2}^2, \dot{t}_{T_2}^2)= (t_v^1, \dot{t}_v^1)$. If $t_{T_1}^1 =t_u^2 < t_{T_2}^2$, as $t^2_s$ is strictly increasing, we would have $u < T_2$ and $(t_u^2, \dot{t}_u^2) \in (t_.^1, \dot{t}_.^1)_{\mathbb R^+}$ which would contradict the definition of $T_2$ as an infimum. Therefore, we have $t_{T_1}^1\geq t_{T_2}^2$ and $t_{T_1}^1 = t_{T_2}^2$ by symmetry. Finally, using the monotonicity of $t_s^i$ again, we conclude that $u=T_2$ and $v=T_1$, hence the coupling (\ref{eqn.couplage}). Now let $h$ be a bounded $\mathcal L_H-$harmonic function. Fix $\omega_2 \in \Omega_2$. The map
$\omega_1 \in \Omega_1 \mapsto T_1(\omega_1,\omega_2)$ is a stopping time for the filtration $\sigma( (t_s^1, \dot{t}_s^1), \; s \leq t)_{t \geq 0}$, and it is finite $\mathbb P_1-$almost surely. By the optional stopping theorem, one has
\[
 h(t_0^1, \dot{t}_0^1) = \mathbb E_1 \left[h (t_{T_1}^1, \dot{t}_{T_1}^1)\right] = \int h (t_{T_1}^1, \dot{t}_{T_1}^1)d \mathbb P_1,
 \]
and integrating against $\mathbb P_2$, we get : 
\[
 h(t_0^1, \dot{t}_0^1) =  \int h (t_{T_1}^1, \dot{t}_{T_1}^1)d \mathbb P_1 \otimes d\mathbb P_2 = \int h (t_{T_1}^1, \dot{t}_{T_1}^1)d \mathbb P.
 \]
In the same way, we have  
\[
 h(t_0^2, \dot{t}_0^2) =  \int h (t_{T_2}^2, \dot{t}_{T_2}^2)d \mathbb P_1 \otimes d\mathbb P_2 = \int h (t_{T_2}^2, \dot{t}_{T_2}^2)d \mathbb P.
 \]
By (\ref{eqn.couplage}), we conclude that $h(t_0^1, \dot{t}_0^1) =h(t_0^2, \dot{t}_0^2)$, \emph{i.e.} the function $h$ is constant. 
\end{proof}

\subsubsection{A Liouville theorem for the spherical subdiffusion}
We now extend the above Liouville theorem to the spherical subdiffusion by using a second coupling argument, namely a mirror coupling argument on the sphere. To simplify the expressions in the sequel, we will denote by  $(e_s)_{s \geq 0} := (t_s, \dot{t}_s, \Theta_s)_{s \geq 0}$ the spherical subdiffusion with values in the space $E:=(0, +\infty) \times [1, +\infty) \times \mathbb S^2$ and by $\mathcal L_{E}$ its the infinitesimal generator acting on smooth functions from $E$ to $\mathbb R$.

\begin{prop}\label{pro.poissonspatial}All bounded $\mathcal L_{E}-$harmonic functions are constant.
\end{prop}

\begin{proof}
Fix $e^1_0= (t_0^1, \dot{t}_0^1, \Theta_0^1)\neq e^2_0=(t_0^2, \dot{t}_0^2, \Theta_0^2)$ in $E$. As in the proof of Proposition \ref{pro.poissontemp}, consider two independent solutions $(t_s^1, \dot{t}_s^1)$ and $(t_s^2, \dot{t}_s^2)$ of Equation (\ref{eqn.ttpoint}), starting from $(t_0^1, \dot{t}_0^1)$ and $(t_0^2, \dot{t}_0^2)$ respectively, which coincide after the shift-coupling times $T_1$ and $T_2$: 
$(t_{T_1+s}^1, \dot{t}_{T_1+s}^1)=(t_{T_2+s}^2, \dot{t}_{T_2+s}^2)$, for $s \geq 0$. 
Let us consider two independent spherical Brownian motions $\widetilde{\Theta}^i$ on $\mathbb S^2$, $i=1,2$, which are independent of the two above temporal diffusions and define for $s \geq 0$ and $i=1, 2$: 
\[
C^i_s:=\int_{0}^{s} \frac{du}{|\dot{t}_u^i|^2-1}, \quad \Theta_s^i := \widetilde{\Theta}^i \left( C^i_s \right).
\]
By Remark \ref{rem.brownienspherique}, the two diffusions $e^i_s:=(t_s^i, \dot{t}_s^i, \Theta_s^i), \,i=1,2$ are solutions of the stochastic differential equations (\ref{eqn.ttpoint}--\ref{eqn.crochet}), let us denote by $\mathbb P_i$ their law, define $\mathbb P:= \mathbb P_1 \otimes \mathbb P_2$ and denote by $\mathbb E$ the associated expectation. Define a new process $({\Theta'}_{s}^2)_{s \geq 0}$, such that ${\Theta'}_{s}^2$ coincides with ${\Theta}_{s}^2$ on the time interval $[0, T_2]$ and such that the future trajectory $({\Theta'}_{s}^2)_{s \geq T_2}$ is the reflection of $({\Theta}_{s}^1)_{s \geq T_1}$ with respect to the median plan between the points $\Theta_{T^1}^1$ and $\Theta_{T^2}^2$, see figure \ref{fig.couplage} below. 

\begin{figure}[ht]
\hspace{5cm}\scalebox{0.8}{\begin{picture}(0,0)%
\includegraphics{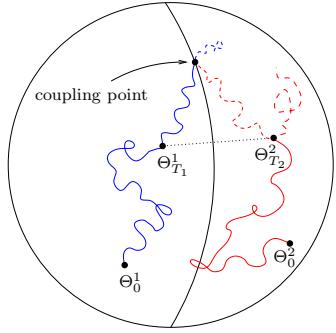}%
\end{picture}}%
\setlength{\unitlength}{1865sp}%
\begingroup\makeatletter\ifx\SetFigFontNFSS\undefined%
\gdef\SetFigFontNFSS#1#2#3#4#5{%
  \reset@font\fontsize{#1}{#2pt}%
  \fontfamily{#3}\fontseries{#4}\fontshape{#5}%
  \selectfont}%
\fi\endgroup%
\scalebox{0.8}{\begin{picture}(5416,5416)(3773,-6819)
\put(5626,-6136){\makebox(0,0)[lb]{\smash{{\SetFigFontNFSS{8}{9.6}{\familydefault}{\mddefault}{\updefault}{\color[rgb]{0,0,0}$\Theta_0^1$}%
}}}}
\put(8191,-5731){\makebox(0,0)[lb]{\smash{{\SetFigFontNFSS{8}{9.6}{\familydefault}{\mddefault}{\updefault}{\color[rgb]{0,0,0}$\Theta_0^2$}%
}}}}
\put(7876,-4021){\makebox(0,0)[lb]{\smash{{\SetFigFontNFSS{8}{9.6}{\familydefault}{\mddefault}{\updefault}{\color[rgb]{0,0,0}$\Theta_{T_2}^2$}%
}}}}
\put(6256,-4156){\makebox(0,0)[lb]{\smash{{\SetFigFontNFSS{8}{9.6}{\familydefault}{\mddefault}{\updefault}{\color[rgb]{0,0,0}$\Theta_{T_1}^1$}%
}}}}
\put(4231,-3031){\makebox(0,0)[lb]{\smash{{\SetFigFontNFSS{8}{9.6}{\familydefault}{\mddefault}{\updefault}{\color[rgb]{0,0,0}coupling point}%
}}}}
\end{picture}}%
\caption{Mirror coupling of two independent spherical sub-diffusions.}\label{fig.couplage}
\end{figure}
\if{\begin{figure}[htbp]
  \begin{center}
    \includegraphics[scale=0.4]{coupling}
    \caption{Mirror coupling of two independent spherical subdiffusions.}
    \label{fig.couplage}
  \end{center}
\end{figure}
}\fi
The new process ${e'}^2_s:=(t_s^2, \dot{t}_s^2, {\Theta'}_s^2)$ is again a solution of Equations (\ref{eqn.ttpoint}--\ref{eqn.crochet}). Moreover, the first time $T^*$ when the process $({\Theta}_{s}^1)_{s \geq T_1}$ intersects the median big circle between $\Theta_{T^1}^1$ and $\Theta_{T^2}^2$ is finite $\mathbb P-$almost surely, and one has naturally ${e'}^2_{T_2+T^*}={e}_{T_1+T^*}^1$  $\mathbb P-$almost surely. Now consider $h$ a bounded $\mathcal L_E-$harmonic function, thanks to the above coupling and the optionnal stopping theorem, as in the proof of Prop. \ref{pro.poissontemp}, we have $\mathbb P-$almost surely $h(e_0^2)=h(e_0^1)$ because 
\[
\mathbb E [  h({e'}^2_{T_2+T^*}) ] =\mathbb E [ h({e}_{T_1+T^*}^1) ].
\]
Therefore, the function $h$ is constant, hence the result.
\end{proof}

\subsubsection{Poisson boundary of the global relativistic diffusion}
\noindent
In order to describe the Poisson boundary of the whole relativistic diffusion $(\xi_s, \dot{\xi}_s)_{s \geq 0}$  starting from the one of the spherical subdiffusion, we need a few preliminaries. First notice that, thanks to the pseudo-norm relation (\ref{eqn.pseudo.eucli}), the invariant sigma field of the whole diffusion  $(\xi_s, \dot{\xi}_s)_{s \geq 0}$ with values in $T^1_+ \mathcal M$ coincides almost surely with the one of the diffusion process $(e_s, x_s)_{s \geq 0}=((t_s, \dot{t}_s, \Theta_s), x_s)_{s \geq 0}$ with values in $E \times \mathbb R^3$ and whose infinitesimal generator $\mathcal G$ is hypoelliptic and reads: 
\begin{equation}\label{gene} \mathcal G:=  \mathcal L_E+F(e) \, \partial_x, \;\; \hbox{where}\;\;  F(e)=F(t,\dot{t}, \Theta):=\Theta \times \frac{\sqrt{\dot{t}^2-1}}{\alpha(t)}.\end{equation}
Without loss of generality, we can suppose that the process $(e_s, x_s)_{s \geq 0}$ is defined on the canonical probability space $(\Omega, \mathcal F)$ where $\Omega:= C(\mathbb R^+,  E \times \mathbb R^3)$ and $\mathcal F$ is the standard Borel sigma field. A generic $\omega \in \Omega$ writes $\omega=(\omega^1, \omega^{2})$ where $\omega^1=(\omega^1_s)_{s \geq 0} \in C(\mathbb R^+,E)$ and $\omega^{2}=(\omega^2_s)_{s \geq 0} \in C(\mathbb R^+,\mathbb R^3)$. Without loss of generality again, we can suppose that $(e_s, x_s)_{s \geq 0}$ is the coordinate process, namely $(e_s, x_s) = (\omega^1_s, \omega^2_s)$ for all $s \geq 0$. Given $(e,x)$ in $E \times \mathbb R^3$, we will denote by $\mathbb P_{(e,x)}$ the law of the process $(e_s, x_s)_{s \geq 0}$ starting from $(e,x)$, and by $\mathbb E_{(e,x)}$ the associated expectation. 
Let us finally introduce the classical shift operators $(\theta_u)_{u \geq 0}$ acting on $\Omega$ and such that $\theta_u \omega = (\omega_{s+u})_{s \geq 0}$ for all $u \geq 0$. Recall that the tail sigma field $\mathcal F^{\infty}$ of the diffusion process $(e_s,x_s)_{s \geq 0}$ is defined as the intersection
\[
 \mathcal F^{\infty}:=\bigcap_{s>0} \sigma( (e_u,x_u), u>s),
 \]
and that the invariant sigma field $\textrm{Inv}((e_s,x_s)_{s \geq 0})$ of $(e_s,x_s)_{s \geq 0}$ is the subsigma field of $\mathcal F^{\infty}$ composed of shift invariant events, i.e. events $A$ such that $\theta^{-1}_u A=A$ for all $u \geq 0$. 
In this setting, Theorem \ref{the.poisson} is equivalent to the following proposition:

\begin{prop}\label{pro.poisson}
Let $h$ be a bounded $\mathcal G-$harmonic function on $E \times \mathbb R^3$. Then, there exists a bounded mesurable function $\psi$ on $\mathbb R^3$ such that:
\[
h(e, x) = \mathbb E_{(e,x)} [\psi(x_{\infty})], \;\; \forall (e,x) \in E \times \mathbb R^3.
\]
Equivalently, $(e_0,x_0)$ being fixed, the invariant sigma field $\textrm{Inv}((e_s,x_s)_{s \geq 0})$ of the diffusion $(e_s,x_s)_{s \geq 0}$ starting from $(e_0,x_0)$ coincides with $\sigma(x_{\infty})$ up to $\mathbb P_{(e_0,x_0)}-$negligeable sets.
\end{prop}

\begin{proof}
From the second point of Theorem \ref{the.asymp}, for all $(e,x) \in E \times \mathbb R^3$, the process $(x_s)_{s \geq 0}$ converges $\mathbb P_{(e,x)}-$almost surely to a random point $x_{\infty} =x_{\infty}(\omega) \in \mathbb R^3$. With a slight abuse of notation, let us still denote by $x_{\infty}$  the random variable which coincides with $x_{\infty}$ on the subset of $\Omega$ where the convergence occurs and which vanishes elsewhere.
Thanks to the particular form (\ref{gene}) of the infinitesimal generator $\mathcal G$, let us remark the following facts:
\begin{enumerate}
\item for all starting points $(e,x) \in E \times  \mathbb R^3$, the law of the process $(e_s, x+ x_s)_{s \geq 0 }$ under $\mathbb P_{(e,0)}$ coincide with the law of $(e_s, x_s)_{s \geq 0 }$ under $\mathbb P_{(e,x)}$, in particular the law of the limit $x_{\infty}$ under $\mathbb P_{(e,x)}$ is the law of $x + x_{\infty}$ under $\mathbb P_{(e,0)}$;
\item the push-forward measures of both measures $\mathbb P_{(e, 0)}$ and $\mathbb P_{(e,x)}$ under the following mesurable map 
$\omega=(\omega^1, \omega^2) \mapsto (\omega^1,  \omega^2-x_{\infty}(\omega))$ coincide.
\end{enumerate}
\noindent
Let $h$ be a bounded $\mathcal G-$harmonic function on $E \times \mathbb R^3$. From the classical duality between harmonic functions and invariant events, there exists a bounded  variable map $Z : \Omega \to \mathbb R$, such that $Z$ is $\mathcal F^{\infty}-$measurable and satisfies $Z( \theta_u \omega) = Z(\omega)$ for all $\omega \in \Omega$, and such that  
\[
h(e,x) = \mathbb{E}_{(e,x)} [ Z ], \quad \hbox{for all} \;\; (e,x) \in E \times \mathbb R^3.\] 
Moreover, $(e, x) \in E \times \mathbb R^3$ being fixed, for $\mathbb P_{(e,x)}-$almost all paths $\omega$, we have: 
\[
Z(\omega)= \lim_{s \to +\infty} h(e_s(\omega), x_s(\omega)).
\]
For $y \in \mathbb R^3$, consider the new random variable 
\[
Z^y(\omega)  := Z(( \omega^1, \omega^2-x_{\infty}(\omega) + y) ).
\]
The variable $Z^y $ is again $\textrm{Inv}((e_s,x_s)_{s \geq 0})-$measurable. Indeed, since the constant function equal to $y$ and the random variable $Z$ are shift-invariant, for all $u \geq 0$ we have 
\[
 Z( (\omega_{.+u}^1, \omega^2_{.+u}-x_{\infty}(\omega_{.+u}) + y) )= Z (\theta_u ( \omega^1, \omega^2-x_{\infty}(\omega) + y ))= Z(  ( \omega^1, \omega^2-x_{\infty}(\omega) + y )).
\]
Since $Z^y$ is bounded and $\textrm{Inv}((e_s,x_s)_{s \geq 0})-$measurable, the function $(e,x) \mapsto \mathbb E_{(e,x)} [Z^{y}]$ is also a bounded $\mathcal G-$harmonic function. But from the point 2 of the beginning of the proof, for all starting points $(e,x, x') \in E \times \mathbb R^3 \times \mathbb R^3$, we have 
$\mathbb{E}_{(e,x)} [ Z^{y}  ]= \mathbb{E}_{(e,x')} [ Z^{y} ]$.
In other words, the harmonic function $(e, x) \mapsto \mathbb{E}_{(e,x)} [ Z^{y}  ]$ is constant in $x$ and its restriction to $E$ is $\mathcal L_E-$harmonic. From Proposition \ref{pro.poissonspatial},  we deduce that the function $(e, x) \mapsto \mathbb{E}_{(e,x)} [ Z^{y}  ]$ is constant. In the sequel, we will denote by $\psi(y)$ the value of this constant. Note that $y \mapsto \psi(y)$ is a bounded measurable function since $y \mapsto Z^y$ is.
Let us now introduce an approximate unity $(\rho_n)_{n \geq 0}$ on $\mathbb R^3$, fix $\mathbf{x} \in \mathbb R^3$, $n \in \mathbb N$ and consider the ``conditionned and regularized'' version $Z$, namely:
$$
Z^{ \mathbf{x},n}(\omega):= \int_{\mathbb R^3} Z^y(\omega) \rho_n( \mathbf{x}-y)dy. 
$$
The same reasoning as above shows that $Z^{ \mathbf{x},n}$ is a bounded $\textrm{Inv}((e_s,x_s)_{s \geq 0})-$measurable variable so that the function  $(e, x) \mapsto \mathbb{E}_{(e,x)} [ Z^{\mathbf{x},n}  ] $ is constant. Hence, for all $\mathbf{x} \in \mathbb R^3$, $n \in \mathbb N$ and $(e,x) \in E \times \mathbb R^3$, there exists a set  $\Omega^{\mathbf{x},n,(e,x)} \subset \Omega$ such that $\mathbb P_{(e,x)}(\Omega^{\mathbf{x},n,(e,x)} ) =1$ and such that for all paths $\omega$ in $\Omega^{\mathbf{x},n,(e,x)}$, we have:
\[
Z^{\mathbf{x},n}  (\omega) = \lim_{s \to \infty}  \mathbb{E}_{(e_s(\omega),x_s(\omega))} [ Z^{\mathbf{x},n}   ] 
= \mathbb{E}_{(e_0(\omega),x_0(\omega))} [ Z^{\mathbf{x},n}   ] 
= \mathbb{E}_{(e,x)} [Z^{\mathbf{x},n}   ].
\]
Let $D$ be a countable dense set in $\mathbb R^3$ and consider the intersection
 \[
 \Omega^{(e,x)} := \underset{\mathbf{x} \in D, n \in \mathbb N}{\bigcap } \Omega^{\mathbf{x},n, (e,x)}.
 \]
We have naturally $\mathbb P_{(e,x)} ( \Omega^{(e,x)} )=1$ and for all $\omega \in \Omega^{(e,x)}$,  $\mathbf{x} \in D$, $ n \in \mathbb N$, we have 
\[
Z^{\mathbf{x},n} (\omega) = \mathbb{E}_{(e,x)} [ Z^{\mathbf{x},n} ].
\]
Since the above expressions are continuous in $\mathbf{x}$, we deduce that the last inequality is true for all $\mathbf{x} \in \mathbb R^3$. In other words, we have shown that for all $\mathbf{x} \in \mathbb R^3$ and $\omega $ in $\Omega^{(e,x)}$:
\[
 Z^{\mathbf{x},n} (\omega)  =\mathbb{E}_{(e,x)} [ Z^{\mathbf{x},n} ]= \int_{\mathbb R^3} \psi(y ) \rho_n(\mathbf{x} -y)dy.
\]
In particular, taking $\mathbf{x}= x_\infty (\omega)$, we obtain that for all $\omega \in \Omega^{(e,x)}$ and for all $n \in \mathbb N$:
\[
Z^{x_{\infty}(\omega),n} (\omega)  = \displaystyle{\int_{\mathbb R^3} Z((\omega^1, \omega^2+y) ) \rho_n(-y)dy }= \displaystyle{\int_{\mathbb R^3} \psi(y+ x_{\infty}(\omega)) \rho_n(-y)dy}.
\]
Taking the integral in $\omega$ with respect to $\mathbb P_{(e,x)}$ on $\Omega^{(e,x)}$, we deduce that for all $ n \in \mathbb N$:
\[
 \displaystyle{\mathbb E_{(e,x)} \left[ Z^{x_{\infty},n} \right] } =   \displaystyle{\int_{\mathbb R^3}  \mathbb{E}_{(e,x)} [ \psi(y+x_{\infty}) ] \rho_n(-y)dy },
 \]
which, from the first point at the beginning of the proof yields
\[
\displaystyle{\int_{\mathbb R^3}  h(e, x+y ) \rho_n(-y)dy } =  \displaystyle{\int_{\mathbb R^3} \mathbb{E}_{(e,x+y)} [ \psi(x_\infty )  ] \rho_n(-y)dy.}
\]
To conclude, recall that the infinitesimal generator of the diffusion is hypoelliptic so that $\mathcal G-$harmonic functions are continuous, hence we can let $n$ go to infinity in the above expressions to get the desired result, namely  $h(e,x) = \mathbb{E}_{(e,x)} [\psi(x_\infty)]$.
\end{proof}

\begin{rmk}
As already noticed at the end of the introduction, the proof of the last proposition is the starting point of the very recent work \cite{AT1} in collaboration with C. Tardif, where our main motivation is to exhibit a general setting and some natural conditions that allow to compute the Poisson boundary of a diffusion starting from the Poisson boundary of a subdiffusion of the original one. Indeed, the main ingredients of the proof above are that the infinitesimal generator $\mathcal G$ acting on $E \times \mathbb R^3$ is equivariant under the action of Euclidean translations and that it is hypoelliptic so that harmonic functions are continuous. The devissage method introduced in \cite{AT1} actually shows that, under similar equivariance and regularity conditions, the scheme of the proof of Proposition \ref{pro.poisson} can be generalized to an abstract setting where $E$ is replaced by any differentiable manifold and $\mathbb R^3$ is replaced by a finite dimensional Lie group or a co-compact homogeneous space.
\end{rmk}

%
\newpage
\bibliographystyle{alpha}

\end{document}